\documentclass[reqno]{amsart}

\usepackage{amsmath,amsfonts,amsthm,amssymb,graphics, paralist}

\newcommand{\g}{\geqslant}

\newcommand{\RR}{\mathbb{R}}
\newcommand{\ZZ}{\mathbb{Z}}
\newcommand{\CC}{\mathbb{C}}
\newcommand{\s}{\mathcal{S}}

\newcommand{\p}{\partial}

\newcommand{\les}{\leqslant}
\newcommand{\lesa}{\lesssim}

\newcommand{\supp}{\text{supp }\,}

\newcommand{\mc}[1]{\mathcal{#1}}
\newcommand{\eref}[1]{(\ref{#1})}

\DeclareSymbolFont{bbold}{U}{bbold}{m}{n}
\DeclareSymbolFontAlphabet{\mathbbold}{bbold}

\theoremstyle{plain}
\newtheorem{theorem}{Theorem}
\newtheorem{proposition}[theorem]{Proposition}
\newtheorem{lemma}[theorem]{Lemma}
\newtheorem{corollary}[theorem]{Corollary}

\theoremstyle{remark}
\newtheorem{remark}{Remark}

\setdefaultenum{(i)}{(a)}{1}{A}
\begin{document}
\title[The CSD equations in the Coulomb Gauge]{A note on the Chern-Simons-Dirac equations in the Coulomb Gauge}

\author{Nikolaos Bournaveas}
\address{Department of Mathematics, University of Edinburgh, Edinburgh EH9 3JE, United Kingdom}
\email{N.Bournaveas@ed.ac.uk}

\author{Timothy Candy}
\address{Department of Mathematics, Imperial College London, London SW7 2AZ, United Kingdom}
\email{T.Candy@imperial.ac.uk}

\author{Shuji Machihara}
\address{Department of Mathematics, Faculty of Education, Saitama University, 255 Shimo-Okubo, Sakura-ku, Saitama City 338-8570, Japan}
\email{matihara@mail.saitama-u.ac.jp}



\date{\today}

\begin{abstract}
We prove that the Chern-Simons-Dirac equations in the Coulomb gauge are locally well-posed from initial data in $H^s$ with $s>\frac{1}{4}$. To study nonlinear Wave or Dirac equations at this regularity generally requires the presence of null structure.  The novel point here is that we make no use of the null structure of the system. Instead we exploit the additional elliptic structure in the Coulomb gauge together with the bilinear Strichartz estimates of Klainerman-Tataru.
\end{abstract}

\maketitle

\section{Introduction}

Chern-Simons gauge theories form an important component of the relativistic theory of planar physics. In particular they are used to model physical phenomena such as the fractional quantum hall effect, and have been well studied by physicists, see for instance \cite{Lopez1991, Deser1982, Cho1992} and the references therein. Mathematically, Chern-Simons terms were first introduced in \cite{Chern1974} in connection with certain geometric invariants. More recently a number of results have appeared studying the properties of various partial differential equations arising in connection with Chern-Simons theories, for instance the Chern-Simons-Higgs equations \cite{Bournaveas2009a, Huh2011, Selberg2012}, the Chern-Simons-Schr\"{o}dinger  equations  \cite{Liu2012, Smith2013}, and the Chern-Simons-Dirac equations \cite{Bournaveas2012a, Cho1992, Huh2007a, Huh2012}.\\

In the current article we study the local well-posedness of the Chern-Simons-Dirac (CSD) equations which are given by
        \begin{equation} \label{eqn CSD short}
                \begin{split}
                    i \gamma^\mu D_\mu \psi &= m \psi \\
                        \frac{1}{2} \epsilon^{\mu \nu \rho} F_{\nu \rho} &= -  J^{\mu}
                \end{split}
        \end{equation}
where the unknowns are the spinor $\psi: \RR^{1+2} \longrightarrow \CC^2$, and the gauge $A_{\mu}: \RR^{1+2} \longrightarrow \RR$, $\mu=0, 1, 2$. Repeated indices are summed over $\mu=0, ..., 2$ and raised and lowered using the Minkowski metric $g = \text{ diag }(1, -1, -1)$, $\epsilon^{\mu \nu \rho}$ is the completely anti-symmetric tensor with $\epsilon^{012}=1$,   $D_\mu = \p_\mu - i A_\mu$ is the covariant derivative, and $F_{\nu \rho} = \p_\nu A_\rho - \p_\rho A_\nu$ denotes the curvature of the connection $A_\mu$. The equations are coupled using the Dirac current $J^\nu = \overline{\psi} \gamma^{\nu} \psi$ where $\overline{\psi}= \psi^\dagger \gamma^0$ is the Dirac adjoint, and $\psi^\dagger$ denotes the conjugate transpose. The Gamma matrices $\gamma^\mu$ are $2 \times 2$ complex matrices which satisfy the relations
            $$ \gamma^\mu \gamma^\nu + \gamma^\nu \gamma^\mu = 2 g^{\mu \nu}I_{2\times 2}, \qquad \big(\gamma^j\big)^\dagger =  - \gamma^j, \qquad \big(\gamma^0\big)^\dagger =  \gamma^0.$$
We take the representation
        $$ \gamma^0 = \begin{pmatrix} 1 & 0\\ 0 & -1 \end{pmatrix}, \qquad \gamma^1 = \begin{pmatrix} 0 & i \\ i & 0 \end{pmatrix}, \qquad \gamma^2=
        \begin{pmatrix}
          0& 1 \\ -1& 0
        \end{pmatrix}. $$
\\
The CSD equations are derived from the Lagrangian
        $$ \mc{L}_{CSD} = \frac{1}{4} \epsilon^{\mu \nu \rho} F_{\mu \nu} A_\rho + \overline{\psi}\big( \gamma^\mu D_\mu - m \big) \psi$$
and solutions $(\psi, A_\mu)$ are \emph{gauge invariant}. Namely if $(\psi, A_\mu)$ is a solution, then $(\psi e^{i \theta}, A_{\mu} + \p_\mu \theta)$ is also a solution for any sufficiently regular map $\theta: \RR^{1+2} \rightarrow \RR$. Thus to obtain a well-posed Cauchy problem, we need to couple the CSD system (\ref{eqn CSD short}) with a choice of gauge. Common choices are the Coulomb gauge $\p_1 A_1 + \p_2 A_2=0$, the Lorenz gauge $\p^\mu A_\mu =0$, and the Temporal gauge $A_0 =0$.

Solutions to the CSD equation also satisfy conservation of charge
        \begin{equation}\label{eqn conservation of charge} \| \psi(t) \|_{L^2_x} = \| \psi(0) \|_{L^2_x} \end{equation}
and if $m=0$, are invariant under the rescaling $(\psi, A_{\mu})(t, x) \mapsto \lambda \big( \psi, A_\mu\big)( \lambda t, \lambda x)$. This rescaling leaves the $L^2$ norm unchanged, and so the CSD equation is \emph{charge critical}. Thus ideally we would like to prove local well-posedness from initial data in $L^2$. This would be particulary interesting in view of the conservation of charge (\ref{eqn conservation of charge}). \\

Recently the local and global well-posedness of Chern-Simons systems has received considerable attention, see for instance \cite{Bournaveas2009a, Bournaveas2012a, Huh2011, Huh2010, Huh2007, Huh2012, Liu2012, Selberg2012}. In particular, it was shown by Huh-Oh \cite{Huh2012} that if we couple the Chern-Simons-Dirac equations with the Lorenz gauge condition $\p_\mu A^\mu=0$, then we have local well-posedness for initial data in $H^s$ with $s>\frac{1}{4}$. This improved earlier work of Huh \cite{Huh2007a} where local well-posedness was obtained for $s>\frac{1}{2}$ in the Coulomb gauge, $s>\frac{5}{8}$ Lorenz gauge, and $s>\frac{3}{4}$ in the Temporal gauge.

 A crucial component in the proof of local well-posedness of Huh-Oh in \cite{Huh2012} was the presence of \emph{null structure}. Here null structure refers to the fact that, from the point of view of bilinear estimates, the nonlinear terms in \eref{eqn CSD short} behave better than generic bilinear terms such as $|\psi|^2$. More precisely, if we consider a nonlinear wave equation of the general form
        \begin{equation}\label{eqn intro wave without null structure} \Box u = u \nabla u \end{equation}
 then in general, we have ill-posedness if $s<\frac{3}{4}$ due to the counterexamples of Lindblad \cite{Lindblad1996}. On the other hand, if we replace the nonlinearity $ u \nabla u$ with a null form such as $Q_{ij}( |\nabla|^{-1}u, u)$ where
        $$ Q_{ij}(u, v)=\p_i u \p_j v - \p_j u \p_i v,$$
then we have well-posedness for $s>\frac{1}{4}$, see for instance \cite{Klainerman2002}. Note that the nonlinearities $u \nabla u$ and $Q_{ij}(|\nabla|^{-1} u, u)$ are roughly of the same ``strength'' in terms of derivatives. Now if we write the CSD equations as a system of nonlinear wave equations, then schematically the CSD equations are of the form \eref{eqn intro wave without null structure}. Thus, at least at first glance, it appears that null structure is essential to get LWP below $\frac{3}{4}$. \\

In the current article we show that, if we couple the system (\ref{eqn CSD short}) with the Coulomb Gauge condition
        \begin{equation}\label{Coulomb gauge condition} \p_1 A_1 + \p_2 A_2=0, \end{equation}
then LWP holds for $s>\frac{1}{4}$. This extends the recent results of Huh-Oh from the Lorenz gauge to the Coulomb gauge. The advantage of the Coulomb gauge is that \emph{no null structure is needed}. This is somewhat surprising in light of the schematic form of the CSD equation, and the counterexamples of Lindblad mentioned above.\\

Our main result is the following.

\begin{theorem}
  \label{main thm intro}
Let $\frac{1}{4}<s<1$ and assume $\psi(0) \in H^s$. Then there exists a $T=T(\| \psi(0)\|_{H^s}) >0$ and a solution $(\psi, A) \in C([0, T], H^s \times \dot{H}^{2s})$ to the CSD equations in the Coulomb gauge. Moreover the solution depends continuously on the initial data, and if we let $I=[0, T]$ then we have
        $$ \| \psi \|_{L^\infty_t H^s_x(I\times \RR^2)} + \| \gamma^\mu \p_\mu \psi \|_{L^1_t H^s_x(I\times \RR^2)} + \|A_\mu \|_{L^\infty_t \dot{H}^{2s}_x(I\times \RR^2)} \lesa \| \psi(0) \|_{H^s(\RR^2)}$$
and the solution is unique in this class.
\end{theorem}

\begin{remark}
  In the result of Huh \cite{Huh2007a}, local well-posedness in the Coulomb gauge was obtained under the conditions $\psi \in H^{\frac{1}{2} + \epsilon}$ and $A_{\mu}(0) \in L^2_x$ where the initial data should satisfy the constraints
        $$ \p_1 A_1 + \p_2 A_2=0, \qquad \p_1 A_2 - \p_2 A_1 = J^0.$$
  This is in contrast to Theorem \ref{main thm intro} where we only provide initial data for the spinor $\psi$. This apparent ambiguity is reconciled by the fact that the initial data for $\psi$, completely determines $A_\mu(0)$ via the constraint equations. Thus there is no need to specify data for the gauge $A_\mu(0)$. See Section \ref{sec structure of csd} below.
\end{remark}

The key observation in the proof of Theorem \ref{main thm intro} is that, the equations for the gauge, coupled with the Coulomb gauge condition, mean that $A_\mu$ satisfies elliptic equations of the form
            $$ \Delta A_\mu = \nabla \psi^2.$$
Note that this is peculiar to the Chern-Simons action, if we have instead couple the Dirac equation with the Maxwell equations, then in the Coulomb gauge we only have an elliptic equation for a \emph{component} of the gauge $A_\mu$. On the other hand, the Chern-Simons action gives sufficiently good control over the curvature of the gauge $A_\mu$, that we have an elliptic equation for the \emph{whole} gauge. The proof is completed by using the bilinear Strichartz estimates of Klainerman-Tataru \cite{Klainerman1999}.

\section{Elliptic Structure}\label{sec structure of csd}

We start by examining the equations for the gauge $A_\mu$, for this we need a little preliminary notation.  Define the ``curl'' $\nabla^\bot= (-\p_2, \p_1)$ and recall the identity
        \begin{equation}\label{eqn laplacian decomp into curl + div} \Delta B = \nabla\big( \nabla \cdot B \big) + \nabla^\bot \big( \nabla^\bot \cdot B \big)\end{equation}
where $B: \RR^2 \rightarrow \CC^2$. Define the projections $\mc{P}_{cf}$, $\mc{P}_{df}$ by
    $$ \mc{P}_{cf} B = \frac{1}{\Delta} \nabla \big( \nabla \cdot B\big), \qquad \mc{P}_{df} B = \frac{1}{\Delta} \nabla^\bot \big( \nabla^\bot \cdot  B \big). $$
It is easy to see that $\mc{P}_{cf} $ and $\mc{P}_{df}$ are orthogonal projections on $L^2(\RR^2)$ and $\nabla \cdot \mc{P}_{df} =  \nabla^\bot \cdot \mc{P}_{cf} = 0$. Let $A=(A_1, A_2)$ denote the spatial component of the gauge $A_\mu$. Then the gauge equations in (\ref{eqn CSD short}) can be written as
        \begin{align*}
          \p_t A - \nabla A_0 &= N\\
            \nabla^\bot A &= -J^0
        \end{align*}
with $N=( - J^2, J^1)^T$. Decompose $A= \mc{P}_{cf} A + \mc{P}_{df} A = A^{cf} + A^{df}$ into divergence free and curl free components. Then the previous equations are equivalent to
    \begin{align*}
       \p_t A^{cf} - \nabla A_0 &= \mc{P}_{cf} N \\
                             \Delta A^{df} &= - \nabla^\bot J_0.
    \end{align*}
Note that, unlike in the Maxwell or Yang-Mills gauge theories, we have an elliptic component \textit{independent of the choice of gauge}. If now enforce the Coulomb gauge condition
        $$ \nabla \cdot A = \nabla \cdot A^{cf} = 0$$
we see that we must have $A^{cf} = 0$ and therefore, the equations for the gauge $(A_0, A)$ are
        \begin{align*}
           \nabla A_0 &= - \mc{P}_{cf} N  \\
           \Delta A^{df} &= - \nabla^\bot J^0.
        \end{align*}
Taking $\nabla^\bot$ of both sides of the equation for $A_0$, and adding the equations for the spinor $\psi$, we see that the CSD equations in the Coulomb gauge are
     \begin{equation}\label{eqn CSD in Coulomb gauge} \begin{split}
      i \gamma^\mu \p_\mu \psi &= m \psi - A_\mu \gamma^\mu \psi\\
       \Delta A_0 &=  \p_1 J_2 - \p_2 J_1 \\
                             \Delta A^{df} &= - \nabla^\bot J^0 \\
                             A^{cf} &= 0.
      \end{split}\end{equation}

\section{Proof of Well-posedness}

Here we prove Theorem \ref{main thm intro}. By taking the equations for the gauge $A_\mu$, and substituting them into the Dirac component, we see that to prove Theorem \ref{main thm intro}, it is enough to prove well-posedness for the cubic Dirac equation
        \begin{equation}  \label{CSD schematically} \begin{split}
                i \gamma^\mu \p_\mu \psi &= m \psi - N(\psi, \psi) \psi\\
                            \psi(0) &= f
                \end{split}
        \end{equation}
where $N$ is the bilinear operator
    \begin{equation}\label{eqn defn of B} N(\psi, \phi) = \frac{1}{\Delta} \Big[ \Big( \p_1 (\overline{\psi} \gamma^2 \phi ) - \p_2 ( \overline{\psi} \gamma^1 \phi) \Big) \gamma^0 + \p_2 (\overline{\psi} \gamma^0 \phi) \gamma^1  - \p_1( \overline{\psi} \gamma^0 \phi) \gamma^2 \Big]. \end{equation}
Once we have the solution $\psi$ to (\ref{CSD schematically}), we then reconstruct the gauge $A_\mu$ by solving the elliptic equations
    \begin{equation}\label{eqn reconstruct gauge}
   \begin{split}
      \Delta A_0 &=  \p_1 J_2 - \p_2 J_1 \\
                             \Delta A^{df} &= - \nabla^\bot J_0 \\
                             A^{cf} &= 0.
   \end{split}
   \end{equation}

The proof of local well-posedness for (\ref{CSD schematically}) will rely on the following bilinear refinement of the classical Strichartz estimates for the wave equation due to Klainerman-Tataru  \cite{Klainerman1999}.

\begin{proposition}[\cite{Klainerman1999}]\label{prop bilinear strichartz}
  Let $u= e^{ it |\nabla| } f$, $v= e^{\pm i t|\nabla| } g$ and $\frac{1}{q} + \frac{1}{2r} = \frac{1}{2}$ with $r<\infty$. Then
      $$ \big\| |\nabla|^{-a} \big( uv\big) \big\|_{L^q_t L^r_x} \lesa \| f \|_{\dot{H}^{s} } \| g \|_{\dot{H}^{s}}$$
  provided $s = \frac{3}{4}\big( 1 - \frac{1}{r}\big) - \frac{a}{2}$ and $0 \les a < 1 - \frac{1}{r}$.
\end{proposition}

This has the following useful consequence.

\begin{corollary}\label{Cor estimate for B}
  Let $\frac{1}{4}<s \les \frac{1}{2} $ and $I\subset \RR$ with $|I|<\infty$. Let $B$ be as in (\ref{eqn defn of B}) and assume that $\psi = e^{\pm it |\nabla|} f$, $\phi = e^{\pm it |\nabla|} g$. Then
        $$ \| N(\psi, \phi) \|_{L^2_t L^\infty_x(I \times \RR^2)} + \big\| |\nabla|^{s+\frac{1}{2}} N(\psi, \phi) \big\|_{L^4_t L^2_x(I \times \RR^2)} \lesa \| f \|_{H^s} \| g \|_{H^s}.$$
\begin{proof}
To obtain the $L^4_t L^2_x$ bound we just note that an application of Proposition \ref{prop bilinear strichartz} with $q=4$, $r=2$ gives\footnote{Whenever we multiply two spinors together, i.e. $\psi \phi$, we really mean $\sum_{ i, j} \psi_i \phi_j$ where $\psi_i$, $\psi_j$ are the components of the spinor.}
        $$ \big\| |\nabla|^{s+\frac{1}{2}} N(\psi, \phi) \big\|_{L^4_t L^2_x (I \times \RR^2)} \lesa \big\| |\nabla|^{s-\frac{1}{2}} (\psi \phi) \big\|_{L^4_t L^2_x} \lesa \| f \|_{H^s} \| g \|_{H^s}. $$
On the other hand, for the $L^2_t L^\infty_x$ bound, we start by writing
        $$ N(\psi, \phi) = P_{<1} N(\psi, \phi) + P_{>1} N(\psi, \phi)$$
where $P_{<1}$ is the projection onto frequencies $|\xi| <1$. To deal with the  low frequency component we use the assumption $|I|<\infty$ to obtain
  \begin{align*} \| P_{<1} N( \psi,  \phi )  \|_{L^2_t L^\infty_x(I\times \RR^2)}
        &\lesa \sum_{\lambda\les 1} \lambda^2 \| P_{\lambda} N( \psi,  \phi) \|_{L^2_t L^1_x(I\times \RR^2)}  \\
        &\lesa_I \sum_{\lambda \les 1} \lambda  \| \psi\|_{L^\infty_t L^2_x} \| \phi\|_{L^\infty_t L^2_x} \\
        &\lesa_I \| f \|_{H^s} \| g \|_{H^s}
    \end{align*}
  where the sum is over dyadic $\lambda \in 2^{\ZZ}$, $\lambda \les 1$, and the $P_{\lambda}$ are the standard Littlewood-Paley projections onto frequencies $|\xi| \approx \lambda$.

    On the other hand, for the high frequency piece we use Sobolev embedding followed by an application of Holder in time to deduce that
        $$ \|  P_{> 1} N( \psi,  \phi) \|_{L^2_t L^\infty_x(I\times \RR^2)}
                \lesa \| |\nabla|^{-a} \big( \psi \phi\big) \|_{L^2_t L^r_x(I\times \RR^2)} \lesa \| |\nabla|^{-a} \big( \psi \phi\big) \|_{L^q_t L^r_x(I\times \RR^2)}$$
  where $a = 1 - \frac{3}{r} < 1 - \frac{2}{r}$ (so we can apply Sobolev embedding) and $q>2$ such that $\frac{1}{q} + \frac{1}{2r} = \frac{1}{2}$ where $r<\infty$ is to be chosen later. An application of Proposition \ref{prop bilinear strichartz} then gives
        $$  \|  P_{> 1} N( \psi,  \phi) \|_{L^2_t L^\infty_x(I\times \RR^2)} \lesa \| f \|_{H^{s'}} \|g \|_{H^{s'}}$$
  where
    $$s'= \frac{3}{4} \Big( 1- \frac{1}{r}\Big) - \frac{a}{2} = \frac{1}{4} + \frac{3}{4r}.$$
  Result now follows by taking $r$ sufficiently large.
\end{proof}
\end{corollary}

 We also require the following version of the product rule for $H^s$.

\begin{proposition}\label{prop modified product rule}
   Let $s>0$ and $\alpha\g0$. Then
        \begin{equation}\label{eqn mod product rule} \| fg \|_{\dot{H}^s} \lesa \| f \|_{L^\infty} \| g \|_{\dot{H}^s} + \| f \|_{\dot{H}^{s + \alpha}} \big\|  |\nabla|^{-\alpha} g \big\|_{L^\infty}.\end{equation}
\end{proposition}
\begin{proof}
  See the Appendix. 
\end{proof}
  The intuition here is that when $g$ is higher frequency than $f$, we should have $|\nabla|^s(fg) \approx f |\nabla|^s g$, which is essentially the first term. On the other hand, when $f$ is higher frequency than $g$, we should be able to shift derivatives from $g$ onto $f$, or $f |\nabla|^\alpha g \lesa (|\nabla|^\alpha f) g$, since it is much worse to have a derivative fall on a high frequency piece rather than a low frequency term. To make this more precise requires a straightforward application of Littlewood-Paley theory. \\

Fix $T>0$. The proof of Theorem \ref{main thm intro} will proceed by the standard iteration argument using the Duhamel norm
        $$ \| \psi \|_{Y^s_T} = \| \psi \|_{L^\infty_t H^s_x(I \times \RR^2)} + \| \gamma^\mu \p_\mu \psi \|_{L^1_t H^s_x(I \times \RR^2)}$$
where $I=[0, T]$. It is easy to see that we have the energy inequality
        $$ \| \psi \|_{Y^s_T} \lesa \| \psi(0) \|_{H^s_x} + \| \gamma^\mu \p_\mu \psi \|_{L^1_t H^s_x(I\times \RR^2)}.$$
Moreover we have the following version of the transference principle.

\begin{lemma}
  \label{lem trans princ}
Let $s \in \RR$ and $1\les q, r \les \infty$. Suppose that we have
        \begin{equation}\label{lem trans princ homogeneous est} \big\| M\big( e^{\pm i t |\nabla|} f \big) \big\|_{L^q_t L^r_x(I \times \RR^2)} \lesa \| f \|_{H^s} \end{equation}
for any $f \in H^s$ where $M$ is a Fourier multiplier acting only on the spatial variable $x \in \RR^2$.
Then for any $\psi \in Y^s_T$ we have
        $$ \| M \psi \|_{L^q_t L^r_x(I \times \RR^2)} \lesa  \| \psi \|_{Y^s_T}.$$
\begin{proof}
  Let $U(t)f$ denote the solution operator for the Dirac equation $  i \gamma^\mu \p_\mu \psi = 0$  with initial data $\psi(0) = f$. An  easy computation shows that $U(t-s) = U(t) U(s)$ and
        $$ U(t) = e^{ i t |\nabla|} L_+  + e^{- i t |\nabla|} L_-$$
  where $L_\pm$ are bounded, time-independent, Fourier multipliers on $H^s$ for all $s \in \RR$. Now
  given any $\psi \in L^\infty_t H^s_x(I\times \RR^2)$ we can write
        $$ \psi = U(t) \psi(0) + \int_0^t  U(t-s)F(s) ds $$
  where $F(s) = i \gamma^\mu\p_\mu \psi$.   Hence using (\ref{lem trans princ homogeneous est}) we obtain
        \begin{align*}
          \| M \psi \|_{L^q_t L^r_x(I \times \RR^2)} &\lesa \sum_{\pm}\big\| M\big( e^{\pm i t |\nabla|} L_\pm \psi(0)\big) \big\|_{L^q_t L^r_x(I \times \RR^2)} + \int_0^T \big\| M \big( e^{\pm i(t-s) |\nabla|}   L_\pm F(s)\big) \big\|_{L^q_t L^r_x(I \times \RR^2)} ds \\
          &\lesa \sum_{\pm} \| L_\pm \psi(0) \|_{H^s_x} + \int_0^T \| L_\pm F(s) \|_{H^s_x} ds \\
          &\les \| \psi \|_{Y^s_T}.
        \end{align*}
\end{proof}
\end{lemma}
\begin{remark}
  A similar argument shows that a multi-linear version of Lemma \ref{lem trans princ} also holds. Thus an estimate of the form
            $$ \|M\big( e^{ \pm_1 i  t |\nabla|} f_1, ... , e^{\pm_m i t|\nabla|} f_m \big)\|_{L^q_t L^r_x(I \times \RR^2)} \lesa \Pi_{j=1}^m \| f_j \|_{H^s_x}$$
  immediately implies that
             $$ \|M \big(\psi_1, ... , \psi_m \big)\|_{L^q_t L^r_x(I \times \RR^2)} \lesa \Pi_{j=1}^m \| \psi_j \|_{Y^s_T}.$$
\end{remark}

 We now come to the proof of local well-posedness for the cubic Dirac equation (\ref{CSD schematically}). A standard iteration argument using the energy inequality, followed by Holder in time, and Lemma \ref{lem trans princ}, shows that to prove local well-posedness for (\ref{CSD schematically}) it is enough to prove the estimate
        \begin{equation} \label{eqn main homogeneous est} \| N(\psi_1, \psi_2) \psi_3 \|_{L^2_t H^s_x (I \times \RR^2)} \lesa_I \Pi_{j=1}^3\| f_j\|_{H^s_x} \end{equation}
where we assume $\psi_j = e^{\pm_j i t |\nabla|} f_j$ is a homogeneous wave with data $f_j \in H^s_x$. To prove (\ref{eqn main homogeneous est}) we start by considering the low frequency case $|\xi|<1$. Then by Corollary \ref{Cor estimate for B}  we obtain
    $$ \| P_{\les 1} N( \psi_1, \psi_2) \psi_3 \|_{L^2_t H^s_x} \lesa \| N( \psi_1, \psi_2) \psi_3 \|_{L^2_{t, x}} \lesa \| N( \psi_1, \psi_2) \|_{L^2_t L^\infty_x} \| \psi_3 \|_{L^\infty_t L^2_x} \lesa \| f_1 \|_{H^s_x} \| f_2 \|_{H^s_x} \| f_3 \|_{H^s_x}.$$
We can now replace the $H^s_x$ norm on the left hand side of (\ref{eqn main homogeneous est}) with the homogeneous version $\dot{H}^s$ and hence, via an application of the Sobolev product rule in Proposition \ref{prop modified product rule} (with $\alpha = \frac{1}{2}$),
we deduce that
    \begin{align*} \| N( \psi_1, \psi_2) \psi_3 \|_{L^2_t \dot{H}^s_x (I \times \RR^2)}
            &\lesa  \big\|  N( \psi_1, \psi_2) \big\|_{L^2_t L^\infty_x(I \times \RR^2)} \big\| |\nabla|^{s} \psi_3 \|_{L^\infty_t L^2_x (I \times \RR^2)} \\
            & \qquad \qquad \qquad +\big\| |\nabla|^{s+\frac{1}{2}} N( \psi_1, \psi_2) \big\|_{L^4_t L^2_x(I \times \RR^2)} \big\| |\nabla|^{-\frac{1}{2}} \psi_3 \|_{L^4_t L^\infty_x (I \times \RR^2)}\\
            &\lesa \| f_1 \|_{H^s_x} \| f_2 \|_{H^s_x} \| f_3 \|_{H^s_x}
    \end{align*}
where we used the bilinear estimates in Corollary \ref{Cor estimate for B}, together with the linear $L^4_t L^\infty_x$ Strichartz estimate. \\

 To complete the proof of Theorem \ref{main thm intro}, it only remains to reconstruct the gauge $A_\mu$ by using (\ref{eqn reconstruct gauge}). To compute the correct regularity for the gauge, note that from (\ref{eqn reconstruct gauge}) we have
        $$ \| A_\mu \|_{L^\infty_t \dot{H}^r} \lesa \| \psi^2 \|_{L^\infty_t \dot{H}^{r-1}}$$
and since we are assuming the spinor $\psi \in L^\infty_t H^s_x$ we need the product estimate
            \begin{equation}\label{eqn product est in hom sobolev norms} \| \psi^2 \|_{\dot{H}^{r-1}} \lesa \| \psi\|_{\dot{H}^s}^2. \end{equation}
The required conditions for product estimates in $\dot{H}^s$ to hold, are given by the following.

\begin{proposition}\label{prop homogeneous product est in sobolev spaces}
  Assume $s_1 + s_2 + s_3 = \frac{n}{2}$ with $s_j + s_k >0$ for $j\not= k$. Then
        $$ \| f g \|_{\dot{H}^{-s_1}(\RR^n)} \lesa \| f \|_{\dot{H}^{s_2}(\RR^n)} \| g \|_{\dot{H}^{s_3}(\RR^n)}.$$
\end{proposition}
  We omit the proof of Proposition \ref{prop homogeneous product est in sobolev spaces} since it is well known. However for the special case that we use below, namely $s_1 = 1-2s$, $s_2=s_3=s$, we note that, provided $0<s<\frac{1}{2}$, the estimate follows by a simple application of Sobolev embedding
        $$ \| fg \|_{\dot{H}^{2s -1}} \lesa \| fg \|_{L^p} \lesa \| f \|_{L^q} \| g \|_{L^q} \lesa \| f \|_{\dot{H}^s} \| g \|_{\dot{H}^s}$$
  where $\frac{1}{p} = \frac{1}{2} + \frac{1-2s}{2}$ and $\frac{1}{p} = \frac{2}{q}$. \\

It we now return to estimating the gauge $A_\mu$, we observe that if we want to put $A_\mu \in L^\infty_t \dot{H}^r$, in light of (\ref{eqn product est in hom sobolev norms}) and Proposition \ref{prop homogeneous product est in sobolev spaces},  we need
        $$ r-1 + 1 = 2s, \qquad \Longrightarrow \qquad r=2s$$
and consequently the correct regularity for the gauge is $A_\mu \in L^\infty_t \dot{H}^{2s}$. Note that this required the assumption $s<1$, if $s\g 1$, then the same argument puts the gauge $A_\mu \in \dot{H}^r$ for $0<r<s+1$. \\

\section*{Appendix - Proof of Proposition \ref{prop modified product rule}}\label{sec proof of product rule}

\begin{proof} The first step is to write
        $$ fg \approx \sum_{\lambda} f_\lambda g_{\ll \lambda} + \sum_\lambda f_\lambda g_\lambda + \sum_{\lambda} f_{\ll \lambda} g_\lambda$$
  where $f_\lambda = P_\lambda f$ (with $\lambda \in  2^\ZZ$) and $f_{\ll \lambda} = \sum_{\mu \ll \lambda} P_\mu f$. Note that we can write $f_{\ll \lambda} = \phi_\lambda * f$ with $\phi \in \s$, $\supp \phi \subset \{ |\xi| \les 2\}$ and $\phi_\lambda(x) = \lambda^2 \phi(\lambda x)$.  \\

  For the high-low piece, we use the fact that the Fourier support of $f_\lambda g_{\ll \lambda}$ is contained inside the annulus $|\xi| \approx \lambda$ and so
      \begin{align*}
        \| \sum_\lambda f_\lambda g_{\ll \lambda} \|_{\dot{H}^s}^2 \approx \sum_\lambda \big(\lambda^s \| f_\lambda g_{\ll \lambda} \|_{L^2} \big)^2\lesa \sum_\lambda \lambda^{2s} \| f_\lambda\|_{L^2}^2 \| g_{\ll \lambda } \|_{L^\infty}^2.
        \end{align*}
  Now we observe that\footnote{We use the fact that $|\nabla|^\alpha \phi \in L^1$ provided $\alpha \g 0$. This is obvious in the case $\alpha=0$. On the other hand if $\alpha>0$, by Holder's inequality, followed by  the  Hausdorff-Young inequality,
    $$ \| |\nabla|^\alpha \phi \|_{L^1_x} \lesa \sum_{|\kappa|\les n} \|x^{\kappa} ( |\nabla|^\alpha \phi) \|_{L^q_x} \lesa \sum_{|\kappa|\les n} \| \p_\xi^\kappa( |\xi|^\alpha \widehat{\phi}(\xi) ) \|_{L^{q'}_\xi}$$
  where we are free to choose any $2<q<\infty$. Thus  it suffices to prove $\p_\xi^\kappa ( |\xi|^{\alpha} \widehat{\phi}(\xi) ) \in L^{q'}_\xi$ for $|\kappa|\les n$. Now using the fact that $\widehat{\phi} \in C^\infty_0$ and $\big|\p^\kappa |\xi|^\alpha \big| \lesa |\xi|^{\alpha-n}$, we see that we require $|\xi|^{\alpha - n} \in L^{q'}(|\xi| \les 1)$ which holds provided $(\alpha - n) q' >-n$. Rearranging we obtain $\alpha > n( 1- \frac{1}{q'}) = \frac{n}{q}$ which holds provided $q$ is sufficiently large.}
        \begin{align*} | g_{\ll \lambda}| &= | \big( |\nabla|^\alpha \phi_\lambda\big) * \big(  |\nabla|^{-\alpha} g \big)|\\
          &\lesa \big\|  |\nabla|^{\alpha} \phi_\lambda \big\|_{L^1} \big\|  |\nabla|^{-\alpha} g \big\|_{L^\infty} \\
          &\approx \lambda^\alpha \big\|  |\nabla|^{-\alpha} g \big\|_{L^\infty}
        \end{align*}
  and consequently
    \begin{align*}
      \| \sum_\lambda f_\lambda g_{\ll \lambda} \|_{H^s} &\lesa \Big( \sum_\lambda \lambda^{2s} \| f_\lambda \|_{L^2}^2 \| g_{\ll \lambda} \|_{L^\infty}^2 \Big)^{\frac{1}{2}} \\
      &\lesa \big\||\nabla|^{-\alpha}  g\big\|_{L^\infty } \Big( \sum_\lambda \lambda^{2(s+\alpha)} \| f_\lambda\|_{L^2}^2 \Big)^{\frac{1}{2}} \\
      &\lesa \| f\|_{\dot{H}^{s+\alpha}} \big\| |\nabla|^{-\alpha} g \big\|_{L^\infty}.
    \end{align*}
  The low-high piece follows an identical argument, essentially just repeat the previous reasoning but replace $f$ with $g$, $g$ with $f$ and take $\alpha=0$.

  Thus it only remains to deal with the high-high case. The key trick is to write
            $$ \sum_{\lambda} P_\mu\big( f_\lambda g_\lambda \big) = \sum_{\lambda \g \frac{\mu}{4}} P_\mu\big( f_\lambda g_\lambda\big)
                            = \sum_{\lambda\g \frac{1}{4} } P_\mu\big( f_{\lambda \mu} g_{\lambda \mu} \big)$$
  where to obtain the last equality we just relabeled our sequence to start at $\frac{1}{4}$ instead of $\mu$ ($\in 2^{\ZZ}$). Now using a similar argument to before
    \begin{align*}
        \| \sum_\lambda f_\lambda g_\lambda \|_{H^s} &= \Big( \sum_\mu \mu^{2s} \| \sum_\lambda P_\mu( f_\lambda g_\lambda ) \|_{L^2}^2\Big)^{\frac{1}{2}} \\
        &\lesa \sum_{\lambda \g 1} \Big( \sum_\mu \mu^{2s}  \| f_{\lambda \mu} g_{\lambda \mu} \|_{L^2}^2\Big)^{\frac{1}{2}} \\
        &\lesa \| f \|_{L^\infty} \sum_{\lambda \g 1} \Big( \sum_{\mu} \mu^{2s} \| g_{\lambda \mu} \|_{L^2}^2 \Big)^{\frac{1}{2}}\\
        &\lesa   \| f \|_{L^\infty} \sum_{\lambda \g 1} \lambda^{-s}  \Big( \sum_{\mu} ( \lambda \mu)^{2s} \| g_{\lambda \mu} \|_{L^2}^2 \Big)^{\frac{1}{2}}\\
        &\lesa \| f \|_{L^\infty} \| g \|_{\dot{H}^s}
    \end{align*}
  where the last line follows by again relabeling the sequence. \\

\end{proof}

\bibliographystyle{amsplain}

\begin{thebibliography}{10}

\bibitem{Bournaveas2009a}
N. Bournaveas, \emph{Low regularity solutions of the
  {C}hern-{S}imons-{H}iggs equations in the {L}orentz gauge}, Electron. J.
  Differential Equations (2009), No. 114, 10.

\bibitem{Bournaveas2012a}
N. Bournaveas, T. Candy, and S. Machihara, \emph{Local and global
  well-posedness for the {C}hern-{S}imons-{D}irac system in one dimension},
  Differential Integral Equations \textbf{25} (2012), no.~7-8, 699 -- 718.

\bibitem{Chern1974}
S.-S. Chern and J. Simons, \emph{Characteristic forms and geometric
  invariants}, The Annals of Mathematics \textbf{99} (1974), no.~1, pp. 48--69
  (English).

\bibitem{Cho1992}
Y.~M. Cho, J.~W. Kim, and D.~H. Park, \emph{Fermionic vortex solutions in
  {C}hern-{S}imons electrodynamics}, Phys. Rev. D (3) \textbf{45} (1992),
  no.~10, 3802--3806.

\bibitem{Deser1982}
S.~Deser, R.~Jackiw, and S.~Templeton, \emph{{Three-dimensional massive gauge
  theories}}, Physical Review Letters \textbf{48} (1982), no.~15, 975--978.

\bibitem{Huh2007a}
H. Huh, \emph{Cauchy problem for the fermion field equation coupled with
  the {C}hern-{S}imons gauge}, Lett. Math. Phys. \textbf{79} (2007), no.~1,
  75--94.

\bibitem{Huh2007}
\bysame, \emph{Local and global solutions of the {C}hern-{S}imons-{H}iggs
  system}, J. Funct. Anal. \textbf{242} (2007), no.~2, 526--549.

\bibitem{Huh2010}
\bysame, \emph{Global solutions and asymptotic behaviors of the
  {C}hern-{S}imons-{D}irac equations in {$\Bbb R^{1+1}$}}, J. Math. Anal. Appl.
  \textbf{366} (2010), no.~2, 706--713.

\bibitem{Huh2011}
\bysame, \emph{Towards the {C}hern-{S}imons-{H}iggs equation with finite
  energy}, Discrete Contin. Dyn. Syst. \textbf{30} (2011), no.~4, 1145--1159.
  \MR{2812958}

\bibitem{Huh2012}
H. Huh and S.-J. Oh, \emph{Low regularity solutions to the
  {C}hern-{S}imons-{D}irac and the {C}hern-{S}imons-{H}iggs equations in the {L}orenz gauge},
  ar{X}iv:1209.3841 (2012).

\bibitem{Klainerman2002}
S. Klainerman and S. Selberg, \emph{Bilinear estimates and
  applications to nonlinear wave equations}, Commun. Contemp. Math. \textbf{4}
  (2002), no.~2, 223--295.

\bibitem{Klainerman1999}
S. Klainerman and D. Tataru, \emph{On the optimal local regularity for
  {Y}ang-{M}ills equations in {${\bf R}^{4+1}$}}, J. Amer. Math. Soc.
  \textbf{12} (1999), no.~1, 93--116.

\bibitem{Lindblad1996}
H. Lindblad, \emph{Counterexamples to local existence for semi-linear wave
  equations}, Amer. J. Math. \textbf{118} (1996), no.~1, 1--16.

\bibitem{Liu2012}
B. Liu, P. Smith, and D. Tataru, \emph{Local wellposedness of
  {C}hern-{S}imons-{S}chr\"odinger}, 	arXiv:1212.1476 (2012).

\bibitem{Lopez1991}
A. Lopez and E. Fradkin, \emph{Fractional quantum {H}all effect and
  {C}hern-{S}imons gauge theories}, Phys. Rev. B \textbf{44} (1991), no.~10,
  5246--5262.

\bibitem{Selberg2012}
S. Selberg and A. Tesfahun, \emph{Global well-posedness of the
  {C}hern-{S}imons-{H}iggs equations with finite energy}, Discrete Contin. Dyn.
  Syst. \textbf{33} (2013), no.~6, 2531--2546.
  
\bibitem{Smith2013}
P. Smith, \emph{On Schr\"{o}dinger maps}, arXiv:1301.7483 (2013)

\end{thebibliography}
\providecommand{\bysame}{\leavevmode\hbox to3em{\hrulefill}\thinspace}
\providecommand{\MR}{\relax\ifhmode\unskip\space\fi MR }
\providecommand{\MRhref}[2]{%
  \href{http://www.ams.org/mathscinet-getitem?mr=#1}{#2}
}
\providecommand{\href}[2]{#2}

\end{document}